\documentclass[a4paper,12pt]{amsart}
\usepackage[utf8x]{inputenc} 
\usepackage{mathrsfs}
\usepackage{amsfonts,amssymb,amscd,amsthm,amsmath,graphicx}
\usepackage{longtable}
\usepackage[all]{xy}

\numberwithin{equation}{section}

\theoremstyle{plain}
\newtheorem{theorem}{Theorem}[section]
\newtheorem{lemma}[theorem]{Lemma}

\newtheorem{proposition}[theorem]{Proposition}

\newtheorem{conjecture}{Conjecture}[section]
\newtheorem{question}{Question}[section]

\def\A{\operatorname{A}}
\def\B{\operatorname{B}}
\def\C{\operatorname{C}}
\def\BC{\operatorname{BC}}
\def\D{\operatorname{D}}

\def\G{\operatorname{G}}

\def\SO{\operatorname{SO}}
\def\Sp{\operatorname{Sp}}
\def\SU{\operatorname{SU}}
\def\U{\operatorname{U}}

\def\Ad{\operatorname{Ad}}

\def\Aut{\operatorname{Aut}}

\def\der{\operatorname{der}}
\def\det{\operatorname{det}}

\def\Hom{\operatorname{Hom}}
\def\id{\operatorname{id}}

\def\Ind{\operatorname{Ind}}

\def\ker{\operatorname{ker}}

\def\Lie{\operatorname{Lie}}
\def\max{\operatorname{max}}

\def\sgn{\operatorname{sgn}}

\def\span{\operatorname{span}}

\newcommand{\bbZ}{\mathbb{Z}}

\newcommand{\bbQ}{\mathbb{Q}}

\newcommand{\frt}{\mathfrak{t}}

\begin{document}

\title{Some new results on dimension datum}
\thanks{}

\author{Jun Yu}
\address{Beijing International Center for Mathematical Research, Peking University, No. 5 Yiheyuan Road, Beijing 100871, China.}
\email{junyu@bicmr.pku.edu.cn}

\begin{abstract}
In this paper we show three new results concerning dimension datum. Firstly, for two subgroups $H_{1}$($\cong
\U(2n+1)$) and $H_{2}$($\cong\Sp(n)\times\SO(2n+2)$) of $\SU(4n+2)$, we find a family of pairs of irreducible representations $(\tau_1,\tau_2)\in\hat{H_{1}}\times\hat{H_{2}}$ such that $\mathscr{D}_{H_1,\tau_1}=\mathscr{D}_{H_2,\tau_2}$. With this
we construct examples of isospectral hermitian vector bundles. Secondly, we show that: $\tau$-dimension data of one-dimensional
representations of a connected compact Lie group $H$ determine the image of homomorphism from $H$ to a given compact Lie group
$G$. Lastly, we improve a compactness result for an isospectral set of normal homogeneous spaces $(G/H,m)$ by allowing the
Riemannian metric $m$ vary, but posing a constraint that $G$ is semisimple.
\end{abstract}

\maketitle

\setcounter{tocdepth}{1}

\tableofcontents

\section{Introduction}

Let $G$ be a compact Lie group. Write $\hat{G}$ for the set of isomorphism classes of irreducible complex linear representations
of $G$, which is a countable set. The {\it dimension datum} of a closed subgroup $H$ is defined by \[\mathscr{D}_{H}: \hat{G}
\rightarrow\mathbb{Z},\quad\rho\mapsto\dim\rho^{H}.\] The dimension datum was first studied by Larsen and Pink in their pioneering
work \cite{Larsen-Pink}, with the motivation of helping determine monodromy groups of $\ell$-adic Galois representations. In the
beginning of the 21st century, Langlands launched a program of ``beyond endoscopy", where he used dimension datum as a key
ingredient in his stable trace formula approach to showing general functoriality (\cite{Langlands}, \cite{Arthur}). Since then
dimension datum catches more attention in the mathematical community. Besides number theory and automorphic form theory, dimension
datum also has applications in differential geometry. For example, it is used to construct the first non-diffeomorphic isospectral
simply-connected closed Riemannian manifolds (\cite{An-Yu-Yu}), which is based on the generalized Sunada's method (\cite{Sunada},
\cite{Pesce}, \cite{Sutton}). In \cite{Yu-dimension}, we classified connected closed subgroups of a given compact Lie group with
the same dimension datum, and characterized linear relations among distinct dimension data. In \cite{Yu-compactness} we showed
that the space of dimension data of closed subgroups in a given compact Lie group is compact.

In this paper, we show several new results concerning dimension datum after previous works \cite{Larsen-Pink}, \cite{An-Yu-Yu},
\cite{Yu-dimension}, \cite{Yu-compactness}. Let $\tau$ be an irreducible representation of $H$, define \[\mathscr{D}_{H,\tau}:
\hat{G}\rightarrow\mathbb{Z},\quad\rho\mapsto\dim\Hom_{H}(\tau,\rho|_{H}),\] and call it the {\it $\tau$-dimension datum} of $H$.
Like for dimension datum, one could again ask about equalities and linear relationes among $\tau$-dimension data. In Section 2
we reduce this to the study of characters associated to sub-root systems and weights. Generalizing the treatment in \cite{An-Yu-Yu}
and \cite{Yu-dimension}, for two subgroups $H_{1}$($\cong\U(2n+1)$) and $H_{2}$($\cong\Sp(n)\times\SO(2n+2)$) of $\SU(4n+2)$, we
find a family of pairs of irreducible representations $(\tau_1,\tau_2)\in\hat{H_{1}}\times\hat{H_{2}}$ such that  $\mathscr{D}_{H_1,\tau_1}=\mathscr{D}_{H_2,\tau_2}$. This
enables us to construct examples of isospectral hermitian vector bundles, which is a generalization of examples of isospectral
manifolds found in \cite{An-Yu-Yu}. In Section 3, we show that: $\tau$-dimension data of one-dimensional representations of a
connected compact Lie group $H$ determine the image of homomorphism from $H$ to a given compact Lie group $G$. This result is a
generalization of the main theorem of \cite{Larsen-Pink} by removing the semisimplicity constraint. In Section 4, we improve a
compactness result for an isospectral set of normal homogeneous spaces $(G/H,m)$ by allowing the Riemannian metric $m$ vary, but
posing a constraint that $G$ is semisimple. We also pose a conjecture concerning an isospectral set of normal homogeneous spaces.

\smallskip

\noindent{\it Acknowledgements.} I would like to thank Jiu-kang Yu and Jinpeng An for helpful communications in the early stage of
this work, and to thank Emilio Lauret for detailed comments and corrections on an early draft of this paper. Thanks to Professor
Richard Taylor for asking me a question which motivated Theorem \ref{T:Taylor}. This research is partially supported by the NSFC Grant
11971036.

\section{The $\tau$-dimension datum of a connected subgroup}

\subsection{Root system and character}\label{SS:RS}

Let $T$ be a torus in $G$. Write $X^{\ast}(T)$ for the weight lattice of $T$. Write $$\Gamma^{\circ}=N_{G}(T)/Z_{G}(T).$$ Choose a
biinvariant Riemannian metric on $G$. Restricting to $T$ it gives a positive definite inner product on the Lie algebra $\frt_0$ of
$T$. Dually, it induces a positive definite inner product on the dual space $\frt_0^{\ast}$. We have $X^{\ast}(T)\subset\mathbf{i}
\frt_{0}^{\ast}$. Multiplying by $-1$ and by restriction, it gives a positive definite inner product on $X^{\ast}(T)$, denoted by
$(\cdot,\cdot)$, which is necessarily $\Gamma^{0}$ invariant.

As in \cite[Def. 2.2]{Yu-dimension}, a root system in the lattice $X^{\ast}(T)$ is a finite subset $\Phi$ satisfying
the following conditions:
\begin{enumerate}
\item[(i)] For any two roots $\alpha\in\Phi$ and $\beta\in\Phi$, the element $\beta-\frac{2(\beta,\alpha)}{(\alpha,\alpha)}\alpha
\in\Phi$.
\item[(ii)] (\textbf{Strong integrality}) For any root $\alpha\in\Phi$ and any weight $\lambda\in X^{\ast}(T)$, the number
$\frac{2(\lambda,\alpha)}{(\alpha,\alpha)}$ is an integer.
\end{enumerate}
As in \cite[Def. 3.1]{Yu-dimension}, set \[\Psi_{T}=\big\{0\neq\alpha\in X^{\ast}(T): \frac{2(\lambda,\alpha)}{(\alpha,\alpha)}
\in\bbZ,\ \forall\lambda\in X^{\ast}(T)\big\}.\] Define $\Psi'_{T}$ as the intersection of sub-root systems of $\Psi_{T}$ which
contain all root systems $R(H,T)$ where $H$ runs through connected closed subgroups $H$ of $G$ with $T$ a maximal torus of $H$.
Defined as above, $\Psi_{T}$ is itself a root system in the lattice $X^{\ast}(T)$, and it contains all root systems in the lattice
$X^{\ast}(T)$; $\Psi'_{T}$ is also a root system in the lattice $X^{\ast}(T)$. Both $\Psi_{T}$ and $\Psi'_{T}$ are necessarily
$\Gamma^{0}$ stable. The following proposition summarizes Prop. 3.3 and Cor. 3.4 in \cite{Yu-dimension}.

\begin{proposition}\label{P:Psi'}
We have $W_{\Psi'_{T}}\subset\Gamma^{0}$, and $\Psi'_{T}$ equals to the union of root systems $R(H,T)$ where $H$ runs over
closed connected subgroups $H$ of $G$ with $T$ a maximal torus of $H$.
\end{proposition}

Choose a positive system $\Psi_{T}^{+}$ of $\Psi_{T}$. For a root system $\Phi$ in the lattice $X^{\ast}(T)$, set \[\delta_{\Phi}
=\frac{1}{2}\sum_{\alpha\in\Phi\cap\Psi_{T}^{+}}\alpha.\] For a root system $\Phi$ in the lattice $X^{\ast}(T)$ and a weight
$\lambda\in X^{\ast}(T)$, set \[A_{\Phi,\lambda}=\sum_{w\in  W_{\Phi}}\sgn(w)[\lambda+\delta_{\Phi}-w\delta_{\Phi}]\in
\bbQ[X^{\ast}(T)].\] For a finite group $W$ between $W_{\Phi}$ and $\Gamma^{\circ}$, set \[F_{\Phi,\lambda,W}=\frac{1}{|W|}
\sum_{\gamma\in W}\gamma(A_{\Phi,\lambda})\in\bbQ[X^{\ast}(T)].\] For a weight $\lambda\in X^{\ast}(T)$ and a finite subgroup
$W$ of $\Gamma^{\circ}$, set \[\chi^{\ast}_{\lambda,W}=\frac{1}{|W|}\sum_{\gamma\in W}[\gamma\lambda]\in\bbQ[X^{\ast}(T)].\]
Then, \[F_{\Phi,\lambda,W}=\sum_{w\in  W_{\Phi}}\sgn(w)\chi_{\lambda+\delta_{\Phi}-w\delta_{\Phi},W}^{\ast}.\] Note that $\chi^{\ast}_{\lambda,W}=\chi^{\ast}_{\lambda',W}$ if and only of $W\lambda=W\lambda'$. Choose a set $\Lambda'$ of representatives
of $W$ orbits in $X^{\ast}(T)$. Then, $\{\chi^{\ast}_{\lambda,W}:\lambda\in\Lambda'\}$ is a basis of $\bbQ[X^{\ast}(T)]^{W}$, the
subspace of $W$ invariant characters on $T$.

\begin{proposition}\label{P:tau-character}
Let $\tau_1\in\widehat{H_1}$ and $\tau_{2}\in\widehat{H_2}$. If $\mathscr{D}_{H_1,\tau_1}=\mathscr{D}_{H_2,\tau_2}$, then $H_1$
and $H_2$ have conjugate maximal tori. Assume that $T$ is a maximal torus of both $H_1$ and $H_2$, write $\Phi_{i}$($\subset
X^{\ast}(T)$) for the root system of $H_{i}$ ($i=1,2$). Then, $$\mathscr{D}_{H_1,\tau_1}=\mathscr{D}_{H_2,\tau_2}\Leftrightarrow F_{\Phi_1,\lambda_1,\Gamma^{\circ}}=F_{\Phi_2,\lambda_2,\Gamma^{\circ}},$$ where $\lambda_{i}$($\in X^{\ast}(T)$) is highest weight
of $\tau_{i}$ ($i=1,2$).
\end{proposition}

\begin{proof}
We first calculate $F_{\Phi}(t)\chi_{\lambda}(t)$, where $H$ is a connected closed subgroup of $G$ with $T$ a maximal
torus of $H$, $\Phi\subset X^{\ast}(T)$ is the root system of $H$, $F_{\Phi}$ is the Weyl product of $H$, and
$\chi_{\lambda}$ is the character of an irreducible representation of $H$ with highest weight $\lambda$. Write $\delta
=\delta_{\Phi}$. The calculation goes as follows, \begin{eqnarray*}&&|W_{\Phi}|F_{\Phi}(t)\chi_{\lambda}(t)\\&=&
\chi_{\lambda}\prod_{\alpha\in\Phi}\big(1-[\alpha]\big)\\&=&\prod_{\alpha\in\Phi^{+}}\big([\frac{-\alpha}{2}]-
[\frac{\alpha}{2}]\big)(\chi_{\lambda}\prod_{\alpha\in\Phi^{+}}\big([\frac{\alpha}{2}]-[\frac{-\alpha}{2}]\big))\\&=
&\big(\sum_{w\in W_{\Phi}}\sgn(w)[-w\delta]\big)\big(\sum_{\gamma\in W_{\Phi}}\sgn(\gamma)[\gamma(\lambda+\delta)]\big)
\\&=&\sum_{w,\gamma\in W_{\Phi}}\sgn(w)\sgn(\gamma)[-w\delta+\gamma(\lambda+\delta)]\\&=&\sum_{\gamma\in W_{\Phi}}
\gamma\big(\sum_{w\in W_{\Phi}}\sgn(w) [\lambda+\delta-w\delta]\big)\\&=&|W_{\Phi}|F_{\Phi,\lambda,W_{\Phi}}.
\end{eqnarray*} Then, $F_{\Phi}(t)\chi_{\lambda}(t)=F_{\Phi,\lambda,W_{\Phi}}$. 
Due to $W_{\Phi}\subset\Gamma^{0}$, we have $$\frac{1}{|\Gamma^{0}|}\sum_{\gamma\in\Gamma^{0}}
\gamma\cdot F_{\Phi,\lambda,W_{\Phi}}=F_{\Phi,\lambda,\Gamma^{\circ}}.$$ Then, a similar argument
as in the proof of \cite[Prop. 3.8]{Yu-dimension} shows the conclusion of the proposition.
\end{proof}

The following proposition can be shown in the way as the proof of \cite[Prop. 3.8]{Yu-dimension}.

\begin{proposition}\label{P:tau-linear}
Given a compact Lie group $G$, let $H_{1},H_2,\dots,H_{s}\subset G$ ($s\geq 2$) be a collection of closed connected
subgroups of $G$. For a set of non-zero constants $c_1,\cdots,c_{s}$, in order for $\sum_{1\leq i\leq s}c_{i}
\mathscr{D}_{H_{i},\tau_{i}}=0$ holds it is necessary and sufficient that: for any torus $T$ of $G$, \[\sum_{1\leq
j\leq t}c_{i_{j}}F_{\Phi_{i_{j}},\lambda_{i_{j}},\Gamma^{\circ}}=0.\] Here $\Gamma^{\circ}=N_{G}(T)/Z_{G}(T),$
$\{H_{i_{j}}: i_1\leq i_2\leq\cdots\leq i_{t}\}$ are all subgroups amongst $\{H_{i}: 1\leq i\leq s\}$ with $H_{i_{j}}$
contains a torus conjugate to $T$, $\Phi_{i_{j}}$ is the root system of $H_{i_{j}}$ with respect to $T$, and
$\lambda_{i}$ is highest weight of $\tau_i$.
\end{proposition}

Similar as for dimension datum, one proposes the following two questions which concern the equalities and linear relations
among $\tau$-dimension data.

\begin{question}\label{Q:FR3}
Given a root system $\Psi$, when $F_{\lambda_1,\Phi_1,\Aut(\Psi)}=F_{\lambda_2,\Phi_2,\Aut(\Psi)}$ for two sub-root systems
$\Phi_1,\Phi_2$ of $\Psi$ and two characters $\lambda_1,\lambda_2$ in the lattice $$\Lambda_{\Psi}=\{\lambda\in\mathbb{Q}\Psi:
\frac{2(\lambda,\alpha)}{(\alpha,\alpha)}\in\mathbb{Z},\forall \alpha\in\Psi\}.$$
\end{question}

\begin{question}\label{Q:FR4}
Given a root system $\Psi$, which linear relates are there among the characters $\{F_{\lambda,\Phi,W_{\Psi}}:\Phi\subset\Psi,
\lambda\in\Lambda_{\Psi}\}$?
\end{question}

Similar as corresponding questions for dimension datum, one may reduce both Question \ref{Q:FR3} and Question \ref{Q:FR4}
to the case that $\Psi$ is an irreducible root system. In this paper we do not intend to solve Questions \ref{Q:FR3}
and \ref{Q:FR4}, but only discuss Question \ref{Q:FR3} in the case that $\Psi$ is an irreducible non-reduced root system of
rank $n$.

\subsection{The case when $\Psi=\BC_{n}$}

There is a nice idea in \cite{Larsen-Pink} which transfers characters $F_{\Phi,0,W_{\BC_{n}}}$ into polynomials. In
\cite{An-Yu-Yu} and \cite{Yu-dimension}, we further find matrix expression for the resulting polynomials. Here, we extend
these to the characters $F_{\Phi,\lambda,W_{\BC_{n}}}$.
Following \cite[Section 7]{Yu-dimension}, we briefly recall the idea of \cite{Larsen-Pink} which identifies
the direct limit of character groups with polynomial ring. Set \begin{eqnarray*}&& \bbZ^{n}:=\bbZ\BC_{n}=
\Lambda_{\BC_{n}}=\span_{\bbZ}\{e_1,e_2,...,e_{n}\},\\&& W_{n}:=\Aut(\BC_{n})=W_{\BC_{n}}=
\{\pm{1}\}^{n}\rtimes S_{n},\\&&\bbZ_{n}:=\bbQ[\bbZ^{n}],\\&& Y_{n}:=\bbZ_{n}^{W_{n}}. \end{eqnarray*} For
$m\leq n$, the injection\[\bbZ^{m}\hookrightarrow\bbZ^{n}:(a_1,...,a_{m})\mapsto(a_1,...,,a_{m},0,...,0)\]
extends to an injection $i_{m,n}:\bbZ_{m}\hookrightarrow\bbZ_{n}$. Define $\phi_{m,n}:\bbZ_{m}\rightarrow
\bbZ_{n}$ by \[\phi_{m,n}(z)=\frac{1}{|W_{n}|}\sum_{w\in W_{n}}w(i_{m,n}(z)).\] Thus $\phi_{m,n}\phi_{k,m}=
\phi_{k,n}$ for any $k\leq m\leq n$ and the image of $\phi_{m,n}$ lies in $Y_n$. Hence $\{Y_{m}:\phi_{m,n}\}$
forms a direct system and we define \[Y=\lim_{\longrightarrow_{n}} Y_{n}.\] Define the map $j_{n}:\bbZ_{n}
\rightarrow Y$ by composing $\phi_{n,p}$ with the injection $Y_{p}\hookrightarrow Y$. The isomorphism
$\bbZ^{m}\oplus\bbZ^{n}\longrightarrow\bbZ^{m+n}$ gives a canonical isomorphism $M: \bbZ_{m}\otimes_{\bbQ}
\bbZ_{n}\longrightarrow\bbZ_{m+n}$. Given two elements of $Y$ represented by $y\in Y_{m}$ and $y'\in Y_{n}$
we define \[yy'=j_{m+n}(M(y\otimes y')).\] This product is independent of the choice of $m$ and $n$ and
makes $Y$ a commutative associative algebra.

The monomials $[e_1]^{k_1}\cdots[e_{n}]^{k_{n}}$ ($k_1,k_2,\cdots,k_{n}\in\bbZ$) form a $\bbQ$ basis of $\bbZ_{n}$,
where $[e_i]^{k_i}=[k_ie_i]\in\bbZ_1$ is a linear character. Hence $Y$ has a $\bbQ$ basis \[e(k_1,k_2,...,k_{n})=
j_{n}([e_1]^{k_1}\cdots[e_{n}]^{k_{n}})\] indexed by $n\geq 0$ and $k_1\geq k_2\geq\cdots\geq k_{n}\geq 0$. Mapping
$e(k_1,k_2,...,k_{n})$ to $x_{k_1}x_{k_2}\cdots x_{k_{n}}$, we get a $\bbQ$ linear map \[E: Y\longrightarrow\bbQ[x_0,
x_1,...,x_{n},...].\] This map $E$ is an algebra isomorphism. Here $x_0=1$ and write as $x_{0}$ for notational
convenience. For any $k_1\geq k_2\geq\cdots\geq k_{n}\geq 0$ (each $k_{i}\in\mathbb{Z}$) and $\lambda=k_1e_1+k_2e_2+
\cdots+k_{n}e_{n}$, one has $$j_{n}(\chi^{\ast}_{\lambda,W_{n}})=e(k_1,k_2,\dots,k_{n})\in Y$$ and
$$E(j_{n}(\chi^{\ast}_{\lambda,W_{n}}))=x_{k_1}x_{k_2}\cdots x_{k_{n}}.$$ Given $f\in\bbQ[x_0,x_1,...]$, set
\[\sigma(f)(x_0,x_1,...,x_{2n},x_{2n+1},...)=f(x_0,-x_1,...,x_{2n},-x_{2n+1},...).\] Then, $\sigma$ is an involutive
automorphism of $\bbQ[x_0,x_1,...]$.

Write $a_{n}(\lambda)$, $b_{n}(\lambda)$, $c_{n}(\lambda)$, $d_{n}(\lambda)$ for the image of $j_{n}(F_{\Phi,\lambda,
W_{n}})$ under $E$ for $\Phi=\A_{n-1}$, $\B_{n}$, $\C_{n}$ or $\D_{n}$, and a weight $\lambda\in\mathbb{Z}^{n}$.
Observe that $a_{n}(\lambda)$, $b_{n}(\lambda)$, $c_{n}(\lambda)$, $d_{n}(\lambda)$ are homogeneous polynomials of
degree $n$ with integer coefficients. Write $b'_{n}(\lambda)=(-1)^{\sum_{1\leq i\leq n}k_{i}}\sigma(b_{n}(\lambda))$.
Define matrices \begin{eqnarray*}&&\qquad\qquad\qquad\qquad\qquad A_{n}(\lambda)=(x_{|k_{j}+i-j|})_{n\times n},\\&&
B_{n}(\lambda)\!=\!(x_{|k_{j}+i-j|}\!-\!x_{|k_{j}+2n+1-i-j|})_{n\times n},\ B'_{n}(\lambda)\!=\!(x_{|k_{j}+i-j|}\!+
\!x_{|k_{j}+2n+1-i-j|})_{n\times n},\\&& C_{n}(\lambda)\!=\!(x_{|k_{j}+i-j|}\!-\!x_{|k_{j}+2n+2-i-j|})_{n\times n},
\ D_{n}(\lambda)\!=\!(x_{|k_{j}+i-j|}\!+\!x_{|k_{j}+2n-i-j|})_{n\times n},\\&& \qquad\qquad\qquad\qquad\qquad
D'_{n}(\lambda)=(y_{i,j})_{n\times n},\end{eqnarray*} where $y_{i,j}=x_{|k_{j}+i-j|}\!+\!x_{|k_{j}+2n-i-j|}$ if
$i,j\leq n-1$, $y_{n,j}=\sqrt{2}x_{|k_{j}+n-j|}$, $y_{i,n}=\frac{\sqrt{2}}{2}(x_{|k_{n}+i-n|}+x_{|k_{n}+n-i|})$ and
$y_{n,n}=x_{|k_{n}|}$.

\begin{lemma}\label{L:polynomial-matrix}
We have \[\det A_{n}(\lambda)=a_{n}(\lambda),\quad \det B_{n}(\lambda)=b_{n}(\lambda),\]
\[\det B'_{n}(\lambda)=b'_{n}(\lambda),\quad \det C_{n}(\lambda)=c_{n}(\lambda),\]
\[\frac{1}{2}\det D_{n}(\lambda)=\det D'_{n}(\lambda)=d_{n}(\lambda).\]
\end{lemma}

\begin{proof}
First consider $\Phi=\A_{n-1}$. Then, $a_{n}(\lambda)=E(j_{n}(A_{\Phi,\lambda}))$, where \[A_{\Phi,\lambda}=\sum_{w\in S_{n}}
\sgn(w)[\lambda+\delta-w\delta]\] with $\delta=(\frac{n}{2}-\frac{1}{2},\frac{n}{2}-\frac{3}{2},\dots,\frac{1}{2}-\frac{n}{2})$.
For a permutation $w\in S_{n}$, one has \[E(j_{n}(\sgn(w)[\lambda+\delta-w\delta]))=\sgn(w)\prod_{1\leq j\leq n}x_{|k_{j}+
\tau(j)-j|},\] which is equal to the term in the expansion of $\det A_{n}(\lambda)$ corresponding to the permutation $w^{-1}$.
Hence, $\det A_{n}(\lambda)=a_{n}(\lambda)$.

Now consider $\Phi=\D_{n}$. Define a new character $\epsilon': W_{n}\rightarrow\{1\}$ by $\epsilon'|_{W_{\D_{n}}}=
\sgn|_{W_{\D_{n}}}$ and $\epsilon'(s_{e_1})=1$. Due to $s_{e_{n}}(\delta_{\D_{n}})=\delta_{\D_{n}}$, one has $$F_{\D_{n},
\lambda,W_{\D_{n}}}=\frac{1}{2}\sum_{w\in W_{n}}\epsilon'(w)\chi^{\ast}_{\lambda+\delta-w\delta,W_{n}}$$ where $\delta=
(n-\frac{1}{2},n-\frac{3}{2},\dots,\frac{1}{2})$. Put $E_{n}=\langle s_{e_{j}}:1\leq j\leq n\rangle\subset W_{n}$. Then,
$W_{n}=S_{n}\ltimes E_{n}$. Then, one shows that: for any given $w\in S_{n}$, \[\sum_{\gamma\in E_{n}}\epsilon'(w\gamma)
E(j_{n}(\chi^{\ast}_{\lambda+\delta-w\gamma\delta,W_{n}}))\] is equal to the term in the expansion of
$\det A_{n}(\lambda)$ corresponding to the permutation $w^{-1}$. Hence, $\frac{1}{2}\det D_{n}(\lambda)=d_{n}(\lambda)$.

The proof for $\det B_{n}(\lambda)=b_{n}(\lambda)$ and $\det C_{n}(\lambda)=c_{n}(\lambda)$ is similar to the proof for
$\frac{1}{2}\det D_{n}(\lambda)=d_{n}(\lambda)$. For these, $W_{\B_{n}}=W_{\C_{n}}=W_{n}$, and we just use the sign function
on $W_{n}$. From $\det B_{n}(\lambda)=b_{n}(\lambda)$, by applying the involutive automorphism $\sigma$ we get
$\det B'_{n}(\lambda)=b'_{n}(\lambda)$. It is clear that $\det D'_{n}(\lambda)=\frac{1}{2}\det D_{n}(\lambda)$. Thus,
$\det D'_{n}(\lambda)=d_{n}(\lambda)$.
\end{proof}



\begin{proposition}\label{P:a=cd-a=bb}
(i) Let $n=2m+1$ be odd, $k_1\geq k_2\geq\cdots\geq k_{n}$, and $k_{n+1-i}+k_{i}=0$ ($\forall i$, $1\leq i\leq m$). Then $$a_{2m+1}(\lambda)=c_{m}(\lambda_1)d_{m+1}(\lambda_2),$$ where $\lambda_{1}=(k_1,\dots,k_{m})$, $\lambda_{2}=(k_1,\dots,
k_{m+1})$.

(ii) Let $n=2m$ be even, $k_1\geq k_2\geq\cdots\geq k_{n}$, and $k_{n+1-i}+k_{i}=0$ ($\forall i$, $1\leq i\leq m$). Then $$a_{2m}(\lambda)=b_{m}(\lambda_1)b'_{m}(\lambda_2),$$ where $\lambda_{1}=\lambda_2=(k_1,\dots,k_{m})$.
\end{proposition}


\begin{proof}
(i) Let $L_{m}=(\delta_{i,m+1-j})_{1\leq i,j\leq m}$, where $\delta_{i,j}$ is the Kronecker symbol. Then, $L_{m}^{2}=I$. The
matrix $A_{2m}(\lambda)$ is of the form $$\left(\begin{array}{cc}X&Y\\L_{m}YL_{m}&L_{m}XL_{m}\\\end{array}\right),$$ where
$X,Y$ are two $m\times m$ matrices. By calculation we have \begin{eqnarray*}&&\frac{1}{2}\left(\begin{array}{cc}I&L_{m}\\
-L_{m}&I\\\end{array}\right)\left(\begin{array}{cc}X&Y\\L_{m}YL_{m}&L_{m}XL_{m}\\\end{array}\right)\left(\begin{array}{cc}I
&-L_{m}\\L_{m}&I\\\end{array}\right)\\&=&\left(\begin{array}{cc}X+YL_{m}&0\\0&L_{m}XL_{m}-L_{m}Y\\\end{array}\right).
\end{eqnarray*} One can check that $X+YL_{m}$ (resp. $X-YL_{m}$) is just the matrix $B'_{m}(\lambda_{2})$ (resp.
$B_{m}(\lambda_{1})$). Thus, $a_{2m}(\lambda)=b_{m}(\lambda_1)b'_{m}(\lambda_2)$ by Lemma \ref{L:polynomial-matrix}.

(ii) The matrix $A_{2m+1}(\lambda)$ is of the form $$\left(\begin{array}{ccc}X&\beta^{t}&Y\\ \alpha&z&\alpha L_{m}\\
L_{m}YL_{m}&\gamma^{t}&L_{m}XL_{m}\\\end{array}\right),$$ where $X,Y$ are two $m\times m$ matrices, $\alpha,\beta,\gamma$
are $1\times m$ vectors. By calculation we have \begin{eqnarray*}&&\frac{1}{2}\left(\begin{array}{ccc}I&&L_{m}\\&\sqrt{2}
&\\-L_{m}&&I\\\end{array}\right)\left(\begin{array}{ccc}X&\beta^{t}&Y\\\alpha&z&\alpha L_{m}\\L_{m}YL_{m}&\gamma^{t}&
L_{m}XL_{m}\\\end{array}\right)\left(\begin{array}{ccc}I&&-L_{m}\\&\sqrt{2}&\\L_{m}&&I\\\end{array}\right)\\&=&\left(
\begin{array}{ccc}X+YL_{m}&\frac{\sqrt{2}}{2}(\beta^{t}+L_{m}\gamma^{t})&0\\\sqrt{2}\alpha&z&0\\0&\frac{\sqrt{2}}{2}
(-L_{m}\beta^{t}+\gamma^{t})&L_{m}XL_{m}-L_{m}Y\\\end{array}\right).\end{eqnarray*} The matrix $$\left(\begin{array}{cc}
X+YL_{m}&\frac{\sqrt{2}}{2}(\beta^{t}+L_{m}\gamma^{t})\\\sqrt{2}\alpha&z\\\end{array}\right)$$ is just
$D'_{m+1}(\lambda_{2})$, and the matrix $X-YL_{m}$ is just $C_{m}(\lambda_1)$. Thus, $a_{2m+1}(\lambda)=c_{m}(\lambda_1)
d_{m}(\lambda_2)$ by Lemma \ref{L:polynomial-matrix}.
\end{proof}

\subsection{Isospectral hermitian vector bundles}

Let $H$ be a closed subgroup of a connected compact Lie group $G$, and $(V_{\tau},\tau)$  be a finite-dimensional
irreducible complex linear representation of $H$ ($V_{\tau}$ is the representation space of $\tau\in\widehat{H}$).
Write $E_{\tau}=G\times_{H}V_{\tau}$ for a $G$-equivariant vector bundle on $X=G/H$ induced from $V_{\tau}$.
As a set, $E_{\tau}$ is the set of equivalence classes in $G\times V_{\tau}$, $$(g,v)\sim(g',v')\Leftrightarrow
\exists x\in H\textrm{ s.t. } g'=gx,\ v'=x^{-1}\cdot v.$$ Write $C^{\infty}(G/H,E_{\tau})$ for the space of smooth
sections of $E_{\tau}$. Then, $$C^{\infty}(G/H,E_{\tau})=(C^{\infty}(G,V_{\tau}))^{H},$$ where $C^{\infty}(G,V_{\tau})$
is the space of smooth functions $f: G\rightarrow V_{\tau}$ and $H$ acts on it through $$(xf)(g)=x\cdot f(gx).$$ The
group $G$ acts on $C^{\infty}(G/H,E_{\tau})$ through $$(g'f)(g)=f(g'^{-1}g).$$ By differentiation, we get an action
of $\mathfrak{g}_{0}=\Lie G$ on $C^{\infty}(G/H,E_{\tau})$, and so an action of the universal enveloping algebra
$U(\mathfrak{g}_{0})$ on $C^{\infty}(G/H,E_{\tau})$. Let $\Delta_{\tau}$ denote the resulting differential operator
on $C^{\infty}(G/H,E_{\tau})$ from the Casimir element in the center of $U(\mathfrak{g}_{0})$. The action
of $\Delta_{\tau}$ on $C^{\infty}(G/H,E_{\tau})$ commutes with the action by $G$, and it is a second order elliptic
differential operator.

Choose an $H$-invariant positive definite inner product $(\cdot,\cdot)$ on $V_{\tau}$ (which is unique up to scalar).
It induces a hermitian metric on $E_{\tau}$ and makes it a hermitian vector bundle. Define a hermitian pairing
$(\cdot,\cdot)$ on $C^{\infty}(G/H,E_{\tau})$ by $$(f_1,f_2)=\int_{G/H}(f_{1}(g),f_{2}(g))d(gH),$$ where $d(gH)$ is
a $G$-equivariant measure on $G/H$ of volume $1$. As $\Delta_{\tau}$ is an elliptic differential operator, any
eigen-function of it in $L^{2}(G/H,E_{\tau})$ is a smooth section. By the Peter-Weyl theorem, \begin{equation}
\label{Eq:PW}L^{2}(G/H,E_{\tau})=\hat{\bigoplus}_{\rho\in\widehat{G}}L^{2}(G/H,E_{\tau})_{\rho}\end{equation} where
$L^{2}(G/H,E_{\tau})_{\rho}$ is the $\rho$-isotropic subspace which has multiplicity
equal to $\dim\Hom_{H}(\tau,\rho|_{H})$ by the Frobenius reciprocity. We know that $\Delta_{\tau}$ acts on the
$\rho$-isotropic component $L^{2}(G/H,E_{\tau})_{\rho}$ by a scalar determined by $\rho$. By this, we have the
following fact: if $\mathscr{D}_{H_1,\tau_1}=\mathscr{D}_{H_2,\tau_2}$, then the Hermitian vector bundles
$E_{\tau_1}=G\times_{H_1}V_{\tau_{1}}$ (on $G/H_1$) and $E_{\tau_2}=G\times_{H_2} V_{\tau_2}$ (on $G/H_2$) are
isospectral with respect to the differential operators $\Delta_{\tau_1}$ and $\Delta_{\tau_2}$.


In $G=\SU(4n+2)$, set $$H_1=\{(A,\overline{A}): A\in\U(2n+1)\},$$ $$H_2=\{(A,B): A\in\Sp(2n), B\in\SO(2n+2)\}.$$ Then,
$H_1\cong\U(2n+1)$, $H_2\cong\Sp(n)\times\SO(2n+2)$. For a sequence of integers $k_1\geq k_2\geq\cdots\geq k_{2n+1}$
with $k_{i}+k_{2n+2-i}=0$ for any $i$, $1\leq i\leq n$, write $\lambda=(k_1,k_2,\dots,k_{2n+1})$ for a weight of
$H_1\cong\U(2n+1)$. Write $\lambda_1=(k_1,\dots,k_{n})$ for a weight of $\Sp(2n)$, $\lambda_2=(k_{1},\dots,k_{n+1})$
for a weight of $\SO(2n+2)$, and $\lambda'=(\lambda_1,\lambda_2)$ for a weight of $H_2$. Write $\tau_{\lambda}$ (resp.
$\tau_{\lambda'}$) for an irreducible representation of $H_1$ (resp. $H_2$) with highest weight $\lambda$ (resp. $\lambda'$).
By Prop. \ref{P:a=cd-a=bb} we have the following theorem.

\begin{theorem}\label{T:isomorphic}
For $G=\SU(4n+2)$, subgroups $H_1,H_2$ and representations $\tau_{\lambda}$ and $\tau_{\lambda'}$ as above, the hermitian
vector bundles $E_{\tau_{\lambda}}=G\times_{H_1}V_{\tau_{\lambda}}$ (on $G/H_1$) and $E_{\tau_{\lambda'}}=G\times_{H_2}
V_{\tau_{\lambda'}}$ (on $G/H_2$) are isospectral with respect to the differential operators $\Delta_{\tau_{i}}$ ($i=1,2$).
\end{theorem}


\section{Generalization of a theorem of Larsen-Pink}\label{SS:Taylor}

A striking theorem of Larsen and Pink (\cite[Thm. 1]{Larsen-Pink}) says that the dimension datum of a connected compact
semisimple subgroup determines the isomorphism class of the subgroup. Fix a connected compact group $H$ (without assuming
semi-simplicity) and consider homomorphisms from it to a connected compact Lie group $G$. We show in the following Theorem
\ref{T:Taylor} that $\tau$-dimension data for one-dimensional representations of $H$ determine the isomorphism class of the
image of a homomorphism. This answers affirmatively a question of Professor Richard Taylor posed to the author during his
stay in IAS in 2013. 

\begin{theorem}\label{T:Taylor}
Let $G,H$ be connected compact Lie groups, and $f_1,f_2: H\rightarrow G$ be two homomorphisms. If
$$\dim((\rho\circ f_1)\otimes\chi)^{H}=\dim((\rho\circ f_2)\otimes\chi)^{H}$$ for any $\rho\in\widehat{G}$
and any $\chi\in\mathcal{X}(H)=\Hom(H,\U(1))$, then $f_1(H)\cong f_2(H)$.
\end{theorem}

\begin{proof}[Proof of Theorem \ref{T:Taylor}]

\noindent{\it The torus case.} To motivate the proof in the general case, we first show Theorem
\ref{T:Taylor} {\it in the case that $H$ is a torus.} First we show $\ker f_{1}=\ker f_2$. Suppose no.
Without loss of generality we assume that $\ker f_{1}\not\subset\ker f_2$. Then, there exists $\chi\in
\mathcal{X}(H)$ such that $\chi|_{\ker f_{1}}\neq 1$ and $\chi|_{\ker f_{2}}=1$. For any $\rho\in
\widehat{G}$, $\rho\circ f_{1}|_{\ker f_{1}}=1$, hence $\dim((\rho\circ f_1)\otimes\chi)^{H}=0$.
As $\chi|_{\ker f_{2}}=1$, $\chi$ descends to a linear character $\chi'$ of $f_{2}(H)\subset G$. Choose
some $\rho\in\widehat{G}$ such that $\rho\subset\Ind_{f_{2}(H)}^{G}(\chi'^{\ast}).$ Then, $\dim((\rho
\circ f_2)\otimes\chi)^{H}>0$. This is in contradiction with $\dim((\rho\circ f_1)\otimes\chi)^{H}=
\dim((\rho\circ f_2)\otimes\chi)^{H}$. Thus, $\ker f_{1}=\ker f_2$.

By considering $H/\ker f_{1}$ instead, we may assume that both $f_{1}$ and $f_{2}$ are injections. By considering the
support of the Sato-Tate measure of $f_{i}(H)$ (which is the push-forward to $G^{\sharp}$ (the space of $G$-conjugacy
classes in $G$) of a normalized Haar measure on $H$ under the map $f_{i}(H)\hookrightarrow G\rightarrow G^{\sharp}$),
we know that $f_{1}(H)$ and $f_{2}(H)$ are conjugate in $G$ (\cite[Prop. 3.7]{Yu-dimension}). We may assume that
$f_{1}(H)=f_{2}(H)$, and denote it by $T$. Write $\Gamma^{\circ}=N_{G}(T)/Z_{G}(T)$.

We identify $H$ with $T$ through $f_1$, and regard $f_{2}$ as an automorphism of $T$, denoted by $\phi$. Then, the
condition in the theorem is equivalent to $$F_{\emptyset,\chi,\Gamma^{0}}=F_{\emptyset,\phi^{\ast}(\chi),\Gamma^{0}}$$
by Prop. \ref{P:tau-linear}. This is also equivalent to $\phi^{\ast}(\chi)\in\Gamma^{\circ}\cdot\chi$. We show
that $\phi=\gamma|_{T}$ for some $\gamma\in\Gamma^{\circ}$. Suppose it is not the case. For any $\gamma\in\Gamma^{0}$,
due to $\phi\neq\gamma^{-1}|_{\Gamma^{0}}$, $$X_{\gamma}=\{\chi\in\mathcal{X}(H):\phi^{\ast}(\chi)=\gamma\cdot\chi\}$$
is a sublattice of $\mathcal{X}(H)$ with positive corank. Hence, $$\bigcup_{\gamma\in\Gamma^{\circ}}X_{\gamma}\neq
\mathcal{X}(H).$$ This is in contradiction with $\phi^{\ast}(\chi)\in\Gamma^{\circ}\cdot\chi$ for any $\chi\in
\mathcal{X}(H)$.

\smallskip

\noindent{\it The general case.} First we show $H_{\der}\ker f_1=H_{\der}\ker f_2$, where $H_{\der}=[H,H]$ is the derived
subgroup of $H$. Suppose no. Without loss of generality we assume that $H_{\der}\ker f_1\not\subset H_{\der}\ker f_2$.
Then, there exists $\chi\in\mathcal{X}(H)$ such that $\chi|_{H_{\der}\ker f_{1}}\neq 1$ and $\chi|_{H_{\der}\ker f_{2}}=1$.
For any $\rho\in\widehat{G}$, $\rho\circ f_{1}|_{\ker f_{1}}=1$, hence $\dim((\rho\circ f_1)\otimes\chi)^{H}=0$. As
$\chi|_{H_{\der}\ker f_{2}}=1$, $\chi$ descends to a linear character $\chi'$ of $f_{2}(H)\subset G$. Choose some $\rho\in
\widehat{G}$ such that $\rho\subset\Ind_{f_{2}(H)}^{G}(\chi'^{\ast}).$ Then, $\dim((\rho\circ f_2)\otimes\chi)^{H}>0$. This
is in contradiction with $\dim((\rho\circ f_1)\otimes\chi)^{H}=\dim((\rho\circ f_2)\otimes\chi)^{H}$. Thus,
$H_{\der}\ker f_1=H_{\der}\ker f_2.$

Write $H_{i}=f_{i}(H)$. Due to $H/H_{\der}\ker f_{i}\cong H_{i}/(H_{i})_{\der}$, we have $$H_{1}/(H_{1})_{\der}
\cong H_{2}/(H_{2})_{\der}.$$ Choose a maximal torus $T_{i}$ of $H_{i}$. Write $(T_{i})_{s}=T_{i}\cap(H_{i})_{\der}$.
Then, $(T_{i})_{s}$ is a maximal torus of $(H_{i})_{\der}$ and $T_{i}=Z(H_{i})^{0}\cdot(T_{i})_{s}.$ Due to
$T_{i}/(T_{i})_{s}\cong H_{i}/(H_{i})_{\der}$, we have $$T_{1}/(T_{1})_{s}\cong T_{2}/(T_{2})_{s}.$$ By considering
the support of Sato-Tate measures of $H_1$ and $H_{2}$, we know that $T_{1}$ and $T_{2}$ are conjugate in $G$
(\cite[Prop. 3.7]{Yu-dimension}). We may assume that $T_{1}=T_{2}$, and denote it by $T$. Write
$\Gamma^{\circ}=N_{G}(T)/Z_{G}(T)$.

Choose a biinvariant Riemannian metric on $G$, which induces a $\Gamma^{\circ}$ invariant inner product on the Lie
algebra of $T$, and also a $\Gamma^{\circ}$ invariant inner product on the weight lattice $X^{\ast}(T)$. Write
$\Phi_{i}\subset X^{\ast}(T)$ for the root system of $H_{i}$. Write $$X_{i}=\mathcal{X}(T_{i}/(T_{i})_{s})\subset
X^{\ast}(T).$$ Then, $T_{1}/(T_{1})_{s}\cong T_{2}/(T_{2})_{s}$ gives an isomorphism $\phi: X_{1}\rightarrow X_{2}$.
For any $\chi_1\in X_1$, write $\chi_2=\phi(\chi_1)$. Then, $$F_{\Phi_1,\chi_1,\Gamma^{0}}=F_{\Phi_2,\chi_2,
\Gamma^{0}}$$ by Prop. \ref{P:tau-linear}. Due to $\chi_{i}$ is orthogonal to $\delta_{\Phi_{i}}-
w\delta_{\Phi_{i}}$ for any $w\in W_{\Phi}$, $\chi_{\chi_{i},\Gamma^{\circ}}^{\ast}$ is the shortest term in the
expansion of $F_{\Phi_{i},\chi_{i},\Gamma^{0}}$. Thus, $\chi_2=\gamma\cdot\chi_{1}$ for some $\gamma\in\Gamma^{\circ}.$
Arguing similarly as in the torus case, one shows that $\phi=\gamma|_{X_{1}}$ for some $\gamma\in\Gamma^{\circ}$.
Replacing $f_{2}$ by $\Ad(g)\circ f_{2}$ for some $g\in N_{G}(T)$ if necessary, we may assume that $\phi=\id$. Then,
$X_{1}=X_{2}$ and $(T_{1})_{s}=(T_{2})_{s}$. As the Lie algebra of $Z(H_{i})^{0}$ is orthogonal to the Lie algebra of
$(T_{i})_{s}$, we have $Z(H_{1})^{0}=Z(H_{2})^{0}$. Write $Z=Z(H_{i})^{0}$, $T_{s}=(T_{i})_{s}$ and $X=X_{i}$.
Let $G'$ be the centralizer of $Z$ in $G$. Put $$\Gamma'=N_{G'}(T_{s})/Z_{G'}(T_{s}).$$ Then, $$\Gamma'=\{\gamma\in
\Gamma^{\circ}: \gamma|_{Z}=\id\}=\{\gamma\in\Gamma^{\circ}: \gamma|_{X}=\id\}.$$

If the rank $X$ has rank $0$ (i.e., $X=0$), then $H_{1}$ and $H_{2}$ are semisimple groups. By \cite[Thm. 1]{Larsen-Pink},
one has $H_{1}\cong H_{2}$. Now assume that $X$ has positive rank. For any $\gamma\in\Gamma^{\circ}-\Gamma'$, $$X_{\gamma}
:=\{\chi\in X:\gamma\cdot\chi=\chi\}$$ is a sublattice of positive corank. Thus, $\bigcup_{\gamma\in\Gamma^{\circ}-\Gamma'}
X_{\gamma}\neq X.$ Choose $$\chi_{0}\in X-\bigcup_{\gamma\in\Gamma^{\circ}-\Gamma'}X_{\gamma}.$$ Write $$c=\min\{|\gamma
\cdot\chi_{0}-\chi_{0}|:\gamma\in\Gamma^{\circ}-\Gamma'\}>0,$$ $$c'=\max\{|\delta_{\Phi_{2}}-w_{2}\delta_{\Phi_2}|+
|\delta_{\Phi_{1}}-w_{1}\delta_{\Phi_1}|: w_{1}\in W_{\Phi_1}, w_{2}\in W_{\Phi_2}\}\geq 0.$$ Take $m\geq 1$ such that
$mc>2c'$. Put $\chi=m\chi_0$. Then, for any $\gamma\in\Gamma^{\circ}$ and any $w_{j}\in W_{\Phi_{i_{j}}}$ ($i_{j}=1$ or 2), $$\gamma(\chi+\delta_{\Phi_{i_{1}}}-w_{1}\delta_{\Phi_{i_{1}}})=\chi+\delta_{\Phi_{i_{2}}}-w_{2}\delta_{\Phi_{i_{2}}}$$ if
and only if $\gamma\in\Gamma'$ and $$\gamma(\delta_{\Phi_{i_1}}-w_{1}\delta_{\Phi_{i_1}})=\delta_{\Phi_{i_2}}-
w_{2}\delta_{\Phi_{i_2}}.$$ Then, $F_{\Phi_1,\chi,\Gamma^{0}}=F_{\Phi_2,\chi,\Gamma^{0}}$ implies $F_{\Phi_1,0,\Gamma'}=
F_{\Phi_2,0,\Gamma'}.$ Define a root system $\Psi_{T_{s}}$ as in the Subsection \ref{SS:RS}. Then, $\Gamma'\subset
\Aut(\Psi_{T_{s}})$. Thus, $$F_{\Phi_1,0,\Aut(\Psi_{T_{s}})}=F_{\Phi_2,0,\Aut(\Psi_{T_{s}})}.$$ By this, results in
\cite[Section 7]{Yu-dimension} imply that $\Phi_2=\gamma\cdot\Phi_1$ for some $\gamma\in \Aut(\Psi_{T_{s}})$. This leads to
an isomorphism $\eta: (H_{1})_{\der}\rightarrow (H_{2})_{\der}$ which stabilizes $T_{s}$ and has $\eta|_{T_s}=\gamma$.
Note that $Z\cap (H_{i})_{\der}=Z\cap(T_{i})_{s}=Z\cap T_{s}\subset T_{s}\cap Z(G')$. Decompose $\Psi_{T_{s}}$ into an
orthogonal union of irreducible root systems, which gives to a decomposition of $T_{s}$. Due to the weight lattice and
the root lattice of a root system $\BC_{n}$ coincide, $T_{s}\cap Z(G')$ is contained in the product of those factors of
$T_{s}$ which correspond to reduced irreducible factors of $\Psi_{T_{s}}$. The results in \cite[Section 7]{Yu-dimension}
imply that there exists $\gamma'\in\Gamma'$ such that the action $\gamma$ on reduced irreducible factors of $\Psi_{T_{s}}$
coincides with that of $\gamma'$. Hence, $$\eta|_{T_{s}\cap Z(G')}=\gamma|_{T_{s}\cap Z(G')}=\gamma'|_{T_{s}\cap Z(G')}=
\id.$$ Then, $\eta$ extends to an isomorphism $\eta: H_1\rightarrow H_2$ by letting $\eta|_{Z}=\id$.
\end{proof}

\section{Compactness of isospectral set}

A big conjecture in spectral geometry says that any set of isospectral closed Riemannian manifolds is compact (\cite{Gordon},
\cite{Osgood-Phillips-Sarnak}). In \cite{Yu-compactness} we show a result of this favor for normal homogeneous spaces.

\begin{theorem}(\cite[Thm. 3.6]{Yu-compactness})\label{T:spectral}
Let $G$ be a compact Lie group equipped with a bi-invariant Riemannian metric $m_0$ and $H$ be a closed subgroup. Then up to
conjugacy, there are finitely many closed subgroups $H_1,\cdots,H_{k}$ of $G$ such that the normal homogeneous space
$(G/H_{j},m_0)$ is isospectral to $(G/H,m_0)$.
\end{theorem}

Recall that in \cite[Thm 1.2]{An-Yu-Yu} we proved that the conjugacy class of a closed subgroup $H$ has only finitely
many possibility if $\mathscr{D}_{H}=\mathscr{D}_{H_0}$, which confirms an expectation of Langlands. Then, in \cite{Yu-compactness}
we proved the above Thm. \ref{T:spectral}, which is stronger than \cite[Thm 1.2]{An-Yu-Yu}. Here we prove a generalization
of Thm. \ref{T:spectral} in case $G$ is semisimple by allowing the Riemannian metric varies.

\begin{theorem}\label{T:spectral2}
Let $G$ be compact semisimple Lie group with a bi-invariant Riemannian metric $m_0$ and $H_{0}$ be a closed subgroup. Then there
are only finitely many conjugacy classes of closed subgroups $H$ of $G$ such that there exists a bi-invariant Riemannian metric
$m$ on $G$ which induces a normal homogeneous space $(G/H,m)$ isospectral to $(G/H_{0},m_{0})$.
\end{theorem}

\begin{proof}
First we may assume that $G$ is connected and simply connected. Write $G=G_1\times\cdots G_{s}$ for the decomposition
of $G$ into simple factors. For each $i$, choose a bi-invariant Riemannian metric $m_{0,i}$ on $G_{i}$. By normalization
we may assume that the Laplace operator and the Casimir operator coincide on $(C^{\infty}(G_{i}),m_{0,i})$ ($1\leq i\leq s$).

Suppose that $\{(G/H_{n},m_{n}): n\geq 1\}$ is a sequence of normal homogeneous spaces such that the Laplace spectrum
of each $(G/H_{n},m_{n})$ is equal to that of $(G/H_{0},m_{0})$, and $H_{n}$ ($n\geq 1$) are non-conjugate to each other.
Write $$m_{n}=\bigoplus_{1\leq i\leq s}a^{(n)}_{i}m_{0,i}.$$ By \cite[Thm. 1.1]{Yu-compactness}, there exists a closed
subgroup $H$ of $G$, a subsequence $\{H_{n_{j}}: j\geq 1\}$ and a sequence $\{g_{j}:j\geq 1, g_{j}\in G\}$ such that
for all $j\in\mathbb{N}$, \[[H^{0},H^{0}]\subset g_{j}H_{n_{j}}g_{j}^{-1}\subset H,\] and \[\lim_{j\rightarrow\infty}
\mathscr{D}_{H_{n_{j}}}=\mathscr{D}_{H}.\] Substituting $\{(G/H_{n},m_{n}): n\geq 1\}$ by a subsequence if necessary we
may assume that: for any $n\geq 1$, \[[H^{0},H^{0}]\subset H_{n}\subset H,\] and \[\lim_{j\rightarrow\infty}
\mathscr{D}_{H_{n}}=\mathscr{D}_{H}.\] Since $H_{n}$ are assumed to be non-conjugate to each other, at most finitely many
of them contain $H^{0}$. By removing such exceptions, we may assume that $\dim H_{n}<\dim H$ for all $n$.

We may also assume that each sequence $\{a_{i}^{(n)}: n\geq 1\}$ converges. Write $$a_{i}=\lim_{n\rightarrow\infty}a_{i}^{(n)}
\in[0,\infty].$$ Without loss of generality we assume that $$a_1=\cdots=a_{u}=0,$$ $$0<a_{u+1},\dots,a_{v}<\infty,$$
$$a_{v+1}=\cdots=a_{s}=\infty,$$ where $0\leq u\leq v\leq s$. Write $$G^{(1)}=\prod_{1\leq i\leq u} G_{i},\quad G^{(2)}=
\prod_{1\leq i\leq v} G_{i},\quad G^{(3)}=\prod_{v+1\leq i\leq s} G_{i},$$ $$G'=\prod_{u+1\leq i\leq v} G_{i},\quad H'=
G'\cap(HG^{(1)}),\quad m'=\bigoplus_{u+1\leq i\leq v}a_{i}m_{0,i}.$$




Write $\chi_{i}(\rho)$ ($1\leq i\leq s$) for the value of the Casimir operator acting on matrix coefficients of
$\rho\in\widehat{G_{i}}$. We know that: $\chi_{i}(\rho)\geq 0$, and $\chi_{i}(\rho)=0$ if and only if $\rho=1$.
We first show that $G^{(3)}\subset HG^{(2)}$. Suppose no. Then, there exists a nontrivial irreducible representation
$$\rho=\bigotimes_{v+1\leq i\leq s}\rho_{i}$$ of $\G^{(3)}$ such that $V_{\rho}^{G^{(3)}\cap HG^{(2)}}\neq 0$. Take
$0\neq v\in V_{\rho}^{G^{(3)}\cap HG^{(2)}}$ and $0\neq\alpha\in V_{\rho}^{\ast}$. Set $$f_{v,\alpha}(g_1,\dots,g_{s})
=\alpha((g_{v+1},\dots,g_{s})\cdot v).$$ Then, $f_{v,\alpha} \in C^{\infty}(G/H)\subset C^{\infty}(G/H_{n})$ for any
$n\geq 1$. The Laplace eigenvalue for $f_{v,\alpha}\in(C^{\infty}(G/H_{n}),m_{n})$ is equal to $$\sum_{v+1\leq i\leq s}
\frac{1}{a_{i}^{(n)}}\chi_{i}(\rho_{i})>0.$$ When $n\rightarrow\infty$, this value tends to $0$. This is in
contradiction with the fact that the Laplace spectrum of each $G/H_{n}$ is equal to a given spectrum which is a
discrete set in $\mathbb{R}_{\geq 0}$.

Now we assume $G^{(3)}\subset HG^{(2)}.$ Then,  $H$ is of the form $$H=(H\cap G^{(2)})\times\{(\phi(x),x):x\in
G^{(3)}\}$$ for some homomorphism $\phi: G^{(3)}\rightarrow G^{(2)}$. Put $$G^{(4)}=\{(\phi(x),x): x\in G^{(3)}\}.$$
Let $G^{(5)}$ be the centralizer of $G^{(4)}$ in $G$. Then, $G^{(5)}\subset G^{(2)}$. Due to $[H^{0},H^{0}]\subset
H_{n}$ for any $n\geq 1$, each $H_{n}$ is of the form $$H_{n}=(H_{n}\cap G^{(2)})\times G^{(4)}.$$ Applying
\cite[Thm. 1.1]{Yu-compactness} to the subgroups $H_{n}\cap G^{(2)}$ of $H\cap G^{(2)}$, we find a subgroup $\tilde{H}$
of $H\cap G^{(2)}$ such that $\lim_{n\rightarrow\infty}\mathscr{D}_{H_{n}\cap G^{(2)}}=\mathscr{D}_{\tilde{H}}$ as
dimension data of subgroups of $H\cap G^{(2)}$. Put $H'=\tilde{H}\times G^{(4)}$. Then, $\lim_{n\rightarrow\infty}
\mathscr{D}_{H_{n}}=\mathscr{D}_{H'}$. Thus, $H'\subset H$ and $\mathscr{D}_{H'}=\mathscr{D}_{H}$. By
\cite[Lemma 2.3]{An-Yu-Yu}, we have $H'=H$. Hence, $\tilde{H}=H\cap G^{(2)}$. Therefore, \[\lim_{n\rightarrow\infty}
\mathscr{D}_{H_{n}\cap G^{(2)}}=\mathscr{D}_{H\cap G^{(2)}}\] as dimension data of subgroups of $G^{(5)}$.

Let $c$ be a positive real number. Suppose matrix coefficients of $$\rho=\bigotimes_{1\leq i\leq s}\rho_{i}$$ contribute to
th Laplace spectrum of $(G/H_{n},m_{n})$ in the eigenvalue scope $[0,c]$. Then, $$\sum_{1\leq i\leq s}\frac{1}{a_{i}^{(n)}}
\chi_{i}(\rho_{i})\leq c$$ and $\rho^{G^{(4)}}\neq 0$. Due to $a_{i}^{(n)}\rightarrow a_{i}$, we have: when $n$ is sufficiently
large, each $\rho_{i}=1$ ($1\leq i\leq u$) and each $\rho_{i}$ ($u+1\leq i\leq s$) lies in a finite set. Due to $\rho^{G^{(4)}}
\neq 0$, $\bigotimes_{v+1\leq i\leq s}\rho_{i}$ is determined by $\bigotimes_{1\leq i\leq v}\rho_{i}$ up to finitely many
possibilities. Then, there are only finitely many $\rho$ in consideration. For each of such $\rho$, we that
\[\lim_{n\rightarrow\infty}\dim V_{\rho}^{H_{n}}=\dim V_{\rho}^{H}=\dim V_{\rho}^{HG^{(1)}}\] for the invariant dimensions,
and \[\lim_{n\rightarrow\infty}\sum_{1\leq i\leq s}\frac{1}{a_{i}^{(n)}}\chi_{i}(\rho_{i})=\sum_{1\leq i\leq s}
\frac{1}{a_{i}}\chi_{i}(\rho_{i})\] for the eigenvalues. Note that \[G/HG^{(1)}\cong G^{(2)}/G^{(2)}\cap HG^{(1)}\cong
G'/G'\cap HG^{(1)}=G'/H'.\] These together imply that: the Laplace spectrum of $(G'/H',m')$ is larger than the Laplace spectrum
of $(G/H_{0},m_{0})$. On the other hand, if matrix coefficients of $$\rho=\bigotimes_{1\leq i\leq s}\rho_{i}$$ contribute to
the Laplace spectrum of $G/HG^{(1)}\cong G'/H'$ in the eigenvalue scope $[0,c]$, then we have the same statements for
$\{\rho_{i}:1\leq i\leq s\}$ as above. By the stabilization of invariant dimensions and the convergence of eigenvalues, it
follows that the Laplace spectrum of $(G'/H',m')$ is smaller than the Laplace spectrum of $(G/H_{0},m_{0})$. Therefore, the
Laplace spectrum of $(G'/H',m')$ is equal to the Laplace spectrum of $(G/H_{0},m_{0})$. By the Minakshisundaram-Pleijel
asymptotic expansion formula, Laplace spectrum determines the dimension (cf. \cite[Subsection 1.1]{Gordon}). Then, $\dim G/H_{n}=
\dim G/H_{0}=\dim G/H$ for any $n\geq 1$. Hence, $\dim H_{n}=\dim H$, which is in contradiction with $\dim H_{n}<\dim H$.
\end{proof}

Motivated by the compactness conjecture of isospectral sets, we think the following statement should hold.

\begin{conjecture}\label{C:finiteness-normal}
There exist only finitely many normal homogeneous spaces $(G/H,m)$ up to isometry with Laplace spectrum equal to a given spectrum.
\end{conjecture}

Recall that for a fixed pair $H\subset G$, different metrics $m$ on $G$ may induce the same metric on $G/H$. When $G$ and
$m$ are both given, Conjecture \ref{C:finiteness-normal} is confirmed affirmatively by Thm. \ref{T:spectral}.
Any normal homogeneous space is of the form $M=G/H$, where $$G=T\prod_{1\leq i\leq s} G_{i}$$ with $T$ a torus and each $G_{i}$
($1\leq i\leq s$) a connected and simply-connected compact simple Lie group, $H\cap T=1$, and $G_{i}\not\subset H$ for any $i$.
Let $M=G/H$ be of this form. When $G$ is semisimple, as $\dim G/H$ is determined by the Laplace spectrum one shows that there
are only finitely many possible $G$. For a fixed $G$, there are only finitely many possible $G/H$ by Theorem \ref{T:spectral2}.
In this case Conjecture \ref{C:finiteness-normal} reduces to the the following question, which has an affirmative answer in
case $G/H$ is a compact symmetric space (cf. \cite{Gordon-Sutton}).

\begin{question}\label{Q:finiteness-normal2}
Let $G$ be a compact semisimple Lie group, and $H$ be a closed subgroup with the above constraint. Are there only finitely many
normal homogeneous spaces $(G/H,m)$ up to isometry with Laplace spectrum equal to a given spectrum?
\end{question}

When $G$ is a torus, then $H=1$ by the above constraint. In this case Conjecture \ref{C:finiteness-normal} is implied by a
theorem of Kneser. A simple proof is given in \cite{Wolpert}, which is based on the Mahler compactness theorem for lattices.

In general, we still have finiteness for $G$ by dimension reason. In this case, the main difficulty is due to the complication
of the invariant inner product on the toric part of the Lie algebra of $G$. Perhaps a sophisticated use of Mahler compactness
theorem coupled with compactness result for dimension datum (\cite[Thm. 1.1]{Yu-compactness}) could overcome this difficulty.

\end{document}